\documentclass[12pt]{article}
\usepackage{authblk}
\usepackage{fixltx2e}
\usepackage[margin=1.3in]{geometry}

\usepackage{amsmath, amsfonts}
\usepackage{color}
\usepackage{amsthm}
\usepackage{amssymb}
\usepackage{color}
\usepackage{epstopdf}
\usepackage{subfigure}
\usepackage[utf8]{inputenc}
\usepackage{hyperref}
\usepackage{cleveref}
\usepackage{bookmark}
\newtheorem{thm}{Theorem}[section]
\newtheorem{cor}[thm]{Corollary}
\newtheorem{lem}[thm]{Lemma}
\newtheorem{prop}[thm]{Proposition}
\theoremstyle{definition}
\newtheorem{example}[thm]{Example}

\let\ol=\overline
\let\ap=\alpha

\newcommand{\Aut}{{\rm Aut}}
\newcommand{\Sub}{{\rm Sub}}
\newcommand{\mbb}{\mathbb}
\newcommand{\N}{{\mathbb N}}
\newcommand{\Z}{{\mathbb Z}}
\newcommand{\R}{{\mathbb R}}
\newcommand{\C}{{\mathbb C}}
\newcommand{\Q}{{\mathbb Q}}
\newcommand{\mH}{{\mathbb H}}
\newcommand{\T}{\overline{T}}
\newcommand{\mT}{\mbb{T}}
\newcommand{\Sc}{{\mathbb S}}

\newcommand{\Id}{{\rm Id}}
\newcommand{\Homeo}{\rm Homeo}

\newcommand{\G}{\mathcal G}
\newcommand{\du}{{\rm d}\,}
\newcommand{\Inn}{{\rm Inn}}
\newcommand{\inn}{{\rm inn}}
\newcommand{\Ad}{{\rm Ad}}

\newcommand{\NC}{{\rm (NC)}}
\newcommand{\GL}{{\rm GL}}
\newcommand{\SL}{{\rm SL}}

\date{}

\title{Distal Actions of Automorphisms of \\ Nilpotent Groups $G$ on Sub$_G$ and \\ Applications to Lattices in Lie Groups}
\author{Rajdip Palit and Riddhi Shah}

\begin{document}
\maketitle
\bigskip
\begin{abstract} For a locally compact group $G$,  we study the distality of the action of automorphisms $T$ of $G$ on $\Sub_G$,   
the compact space of closed subgroups of $G$ endowed with the Chabauty topology. For a certain class of discrete groups $G$, 
we show that $T$ acts distally on $\Sub_G$ if and only if $T^n$ is the identity map for some $n\in\N$. As an application, 
we get that for a $T$-invariant lattice $\Gamma$ in a simply connected nilpotent Lie group $G$, $T$ acts distally on 
$\Sub_G$ if and only if it acts distally on $\Sub_\Gamma$. This also holds for any closed $T$-invariant co-compact subgroup 
$\Gamma$ in $G$. For a lattice $\Gamma$ in a simply connected solvable Lie group, we study conditions under which its automorphisms 
act distally on $\Sub_\Gamma$. We construct an example highlighting the difference between the behaviour of automorphisms on 
a lattice in a solvable Lie group from that in a nilpotent Lie group. We also characterise automorphisms of a lattice $\Gamma$ in
a connected semisimple Lie group which act distally on $\Sub_\Gamma$. For torsion-free compactly generated nilpotent (metrizable) 
groups $G$, we obtain the following characterisation: $T$ acts distally on $\Sub_G$ if and only if $T$ is contained in a compact subgroup 
of  $\Aut(G)$. Using these results, we characterise the class of such groups $G$ which act distally on $\Sub_G$. We also show that any 
compactly generated distal group $G$ is Lie projective.  
\end{abstract}

\quad {\bf Keywords}: Distal actions, Space of closed subgroups, Chabauty topology, \\
\hphantom{2.1 cm}Nilpotent groups, Lattices in Lie groups.

\medskip
\quad {\bf Mathematical Subject Classification (2010)}: Primary 37B05; \\
\hphantom{2.1.cm}Secondary 22E25, 22E40, 22D45. 

\section{Introduction}

Distal actions were introduced by David Hilbert to study the dynamics of non-ergodic actions on compact spaces 
(cf.\ Moore \cite{Mo}). Let $X$ be a (Hausdorff) topological space. A semigroup $\mathfrak{S}$ of homeomorphisms 
of $X$ is said to act {\it distally} on $X$ if for every pair of distinct elements $x, y\in X$, the closure of 
$\{(T(x),T(y))\mid T\in\mathfrak{S}\}$ does not intersect the diagonal $\{(d,d)\mid d\in X \}$. Let $\Homeo(X)$ 
denote the set of homeomorphisms of $X$. The map $T\in\Homeo(X)$ is said to be {\it distal} if the group 
$\{T^n\}_{n\in\Z}$ acts distally on $X$. If $X$ is compact, then $T$ is distal if and only if the semigroup 
$\{T^n\}_{n\in\N}$ acts distally (cf.\ \cite{BJM}). Let $G$ be a locally compact (Hausdorff) group with the identity $e$ 
and let $T\in \Aut(G)$. Then $T$ is distal if and only if 
$e\notin \overline{\{T^{n}(x)\mid n\in{\Z}\}}$ whenever $x\neq e$.

Distal actions on compact spaces have been studied extensively by Ellis \cite{El} who obtained a characterisation, 
and by Furstenberg \cite{Fu} who has a deep structure theorem for distal maps on compact metric spaces. Distal 
actions by automorphisms on Lie groups and locally compact groups have been studied by many mathematicians 
(see Abels \cite{A1, A2}, Jaworski-Raja \cite{JR}, Raja-Shah \cite{RS1,RS2}, Reid \cite{Re}, Shah \cite{Sh} and the 
references cited therein).

A locally compact (Hausdorff) group $G$ is said to be {\it distal} if the conjugacy action of $G$ on $G$ is distal. 
Equivalently, $e\notin \overline{\{gxg^{-1}\mid g\in G\}}$, for every $x\neq e$. All discrete groups, compact groups 
and nilpotent groups are distal. It is well known that a connected locally compact group $G$ is distal if and and only if 
it has polynomial growth; and such a $G$ is a compact extension of a connected solvable normal subgroup 
(see \cite{Ro} and \cite{Je}). In \cite{Re}, Reid has shown that any compactly generated totally disconnected distal 
group is Lie projective. We extend this to all compactly generated locally compact distal groups (see 
Theorem \ref{distal-str}). 

For a locally compact group $G$, let $\Sub_G$ denote the space of all closed subgroups of $G$ equipped with 
the Chabauty topology (cf.\ \cite{Ch}). Then $\Sub_G$ is compact. It is metrizable if $G$ is so (cf.\ \cite{BP}). Note 
that $\Sub_G$ has been identified for certain groups $G$, e.g.\ $\Sub_\R$ is isomorphic to $[0,\infty]$, $\Sub_\Z$ is 
isomorphic to $\{0\}\cup\{1/n\mid n\in\N\}$ and $\Sub_{\R^2}$ is isomorphic to $\Sc^4$. For the study of various aspects
of $\Sub_G$ for different groups $G$, we refer the reader to  Abert et al \cite{ABBGNRS}, Baik and Clavier \cite{BC1,BC2}, 
Bridson et al \cite{BHK}, Pourezza and Hubbard \cite{PH} and the references cited therein.

 There is a natural action of $\Aut(G)$, the group of automorphisms of $G$, on $\Sub_G$ as follows:
$$\Aut(G)\times \Sub_G\to \Sub_G,\ H\mapsto T(H),\ T\in\Aut(G), H\in\Sub_G.$$
Each $T\in \Aut(G)$ defines a homeomorphism of $\Sub_G$ and the corresponding map from 
$\Aut(G)\to\Homeo(\Sub_G)$ is a group homomorphism. 

For automorphisms $T$ of connected Lie groups $G$, Shah and Yadav in \cite{SY3} have studied and characterised 
the distality of the $T$-action on $\Sub_G$ under certain conditions on $T$ or on $G$, e.g.\ $T$ is unipotent or the 
largest connected central subgroup of $G$ is torsion-free. Our main aim is to study the distality of this action for 
some disconnected metrizable groups $G$, namely, a certain class of discrete groups, compact groups and compactly 
generated nilpotent groups.

A discrete (closed) subgroup $\Gamma$ in a locally compact group $G$ is said to be a {\it lattice} in $G$ if $G/H$ carries 
a finite $G$-invariant measure. We refer the reader to \cite{Rag} for generalities on lattices. 
Note that a lattice $\Gamma$ in a simply connected nilpotent group $G$ is a finitely generated 
discrete nilpotent co-compact subgroup and any automorphism $T$ of $\Gamma$ 
extends to a unique automorphism 
of $G$ (cf.\ \cite{Rag}, Theorem 2.11 and Corollary 1 following it). Note also that $\Sub_\Gamma$ is much smaller than 
$\Sub_G$; e.g.\ $\Sub_{\Z^n}$ is countable, while $\Sub_{\R^n}$ is not. We are motivated by the following question: 
Whether it is enough to study the $T$-action on $\Sub_\Gamma$ to determine the distality of the $T$-action on 
$\Sub_G$. We show that it is in fact enough to assume the distality of the $T$-action on $\Sub^c_\Gamma$, the set of 
cyclic subgroups of $\Gamma$, to show that $T^n=\Id$, the identity map on $G$ and hence, it acts distally on $\Sub_G$ 
(more generally, see Corollary \ref{lattice-nilp}). We also get a suitable generalisation of this for a closed co-compact 
Lie subgroup $\Gamma$ in $G$ (see Corollary \ref{closed-nil}). For a lattice $\Gamma$ in a simply connected solvable group 
$G$, if $T\in\Aut(G)$ and $T(\Gamma)=\Gamma$, the distality of the action of $T$ on $\Sub_\Gamma$ implies that 
$T^n|_\Gamma=\Id$, but $T$ need not act distally on $\Sub_G$ (more generally, see Theorem \ref{lattice-solv} and 
Example \ref{ex1}). We also get a characterisation for automorphisms of a lattice $\Gamma$ in a connected 
semisimple Lie group which act distally on $\Sub_\Gamma$ (see Theorem \ref{newp}). For  locally compact 
compactly generated nilpotent (metrizable) groups $G$ such that $G^0$ is torsion-free, we get that 
$T\in\Aut(G)$ acts distally on $\Sub_G$ if and only if $T$ is contained in a compact subgroup of $\Aut(G)$ 
(see Theorem \ref{cg-nilp-distal}). This also holds for any compact totally disconnected metrizable group (see 
Proposition \ref{cpt-distal}). We also characterise  metrizable locally compact compactly generated nilpotent groups 
$G$ whose inner automorphisms act distally on $\Sub_G$ (see Theorem \ref{distal-gp}); this is an analogue of 
Corollary 4.5 of \cite{SY3}. 

Some of the results about the actions of automorphisms on $\Sub_G$ are proven under weaker assumptions 
such as either $T$ acts distally on $\Sub^a_G$, the set of closed abelian subgroups of $G$, or on a smaller class 
$\Sub^c_G$, the set of closed cyclic subgroups of $G$. In \cite {BHK}, Bridson, de la Harpe and Kleptsyn describe 
the structure of $\Sub^a_\mH$ and various other subspaces of $\Sub^a_\mH$, for the 3-dimensional Heisenberg 
group $\mH$, and they also study and describe the action of $\Aut(G)$ on some of these spaces in detail. 
Baik and Clavier have identified $\Sub^a_G$ for $G={\rm PSL}(2,\C)$ in \cite{BC2} and, they also give a description 
of the space which is the closure of $\Sub^c_G$ in $\Sub_G$, where $G$ is either ${\rm PSL}(2,\R)$ or 
${\rm PSL}(2,\C)$ (cf.\ \cite{BC1, BC2}). We give conditions on discrete groups $G$ under which 
$\Sub^c_G$ is closed in $\Sub_G$ and study the distality of the action of automorphisms of $G$ on $\Sub^c_G$. 
We also prove certain results for the automorphisms in the class $\NC$ introduced in \cite{SY3}, which contains 
those that act distally on $\Sub^a_G$ or on the closure of $\Sub^c_G$.

Throughout, let $G$ be a locally compact (Hausdorff) group with the identity $e$. For a subgroup $H$ of $G$, let $H^0$ 
denote the connected component of the identity $e$ in $H$, $[H,H]$ denote the commutator subgroup of $H$, $Z(H)$ 
denote the center of $H$ and let $Z_G(H)$ denote the centraliser of $H$ in $G$. For $B\subset G$, let $\ol{B}$ denote 
the closure of $B$ in $G$. If $B$ is a group, so is $\ol{B}$. For any $T\in\Aut(G)$, $T^0$ is the identity map of $G$. 

\section{Compactly Generated Distal Groups}

Recall that a locally compact group $G$ is distal if the conjugation action of $G$ on $G$ is distal, i.e.\ for every $x\in G$ 
such that $x\ne e$, the closure of $\{gxg^{-1}\mid g\in G\}$ does not contain the identity $e$. Compact groups, nilpotent 
groups and discrete groups are all distal. A locally compact group is said to be Lie projective if it has compact normal 
subgroups $K_\alpha$, such that $\bigcap_\ap K_\ap=\{e\}$ and $G/K_\alpha$ is a Lie group for each $\alpha$. Note 
that any connected, more generally any almost connected locally compact group is Lie projective; ($G$ is 
almost connected if $G/G^0$ is compact). 

It is shown by Willis in \cite{Wi}, that any compactly generated totally disconnected locally compact nilpotent group 
is Lie projective. This was extended by Reid to all compactly generated totally disconnected locally compact distal 
groups (cf.\ \cite{Re}, Corollary 1.9). We generalise this to all locally compact groups as follows. 

\begin{thm} \label{distal-str}
Any compactly generated locally compact distal group is Lie projective. 
\end{thm}

\begin{proof} Let $G$ be a compactly generated locally compact distal group. By Corollary 3.4 of \cite{RS1}, 
$G/G^0$ is distal. Since $G/G^0$ is also compactly generated, by Corollary 1.9 of \cite{Re}, there exists a 
neighbourhood basis of the identity consisting of compact open normal subgroups in $G/G^0$, and hence 
there exist open normal subgroups $H_\alpha$ in $G$, such that $H_\alpha/G^0$ is compact and 
$\bigcap_\alpha H_\alpha= G^0$. Let $\ap$ be fixed. Note that the maximal compact normal subgroup $K$ 
of $H_\ap$ is characteristic in $H_\ap$ and hence normal in $G$. Let $H=KG^0$. Then $H$ is normal in $G$, 
$K$ is the maximal compact normal subgroup of $H$ and $H/K$ is a Lie group. As $H_\ap$ is Lie projective, 
we have that $H$ is an open normal subgroup of $G$. Therefore, $G/H$ is discrete, and it is finitely generated, 
since $G$ is compactly generated. Let $x_1,\ldots, x_n\in G$ be such that their images in $G/H$ generate $G/H$. 
Let $L$ be the subgroup generated by $x_1,\ldots, x_n$ in $G$. Then $L$ is countable. Since the conjugation 
action of $L$ on $K$ is distal, $K$ has compact normal subgroups $K_\beta$ such that $K_\beta$ is $L$-invariant 
and $K/K_\beta$ is a Lie group for each $\beta$, and $\bigcap_\beta K_\beta=\{e\}$ (cf.\ \cite{Ja}, Theorem 2.6 
and Corollary 2.7). Let $\beta$ be fixed. Since $G^0$ normalises $K$, by Theorem $1'$ of \cite{Iw}, the action of 
$G^0$ (by inner automorphisms) on $K$ is the same as the conjugation action of $K^0$ on $K$. Therefore, 
every normal subgroup of $K$ is 
normalised by $G^0$. In particular,  $K_\beta$ is normal in $H=KG^0$. Since $L$ also normalises 
$K_\beta$ and $LH=G$, we get that $K_\beta$ is normal in $G$. As $K/K_\beta$ and $H/K$ are Lie groups, 
so is $H/K_\beta$. Moreover, $G/H$ is discrete. Therefore, $G/K_\beta$ is a Lie group. Since this is true for every  
$\beta$ and since $\bigcap_\beta K_\beta=\{e\}$, $G$ is Lie projective.  
\end{proof}

Note that in Theorem \ref{distal-str},  both the conditions that the group is compactly generated and distal are necessary. 
Willis in \cite{Wi} has given an example of a locally compact nilpotent (distal) group which is not Lie projective. For 
a compactly generated locally compact group which is not distal, one can take $G=\Z\ltimes (\mT^2)^\Z$, where 
$\mT^2$ is the (compact) 2-dimensional torus, and the action of $1\in\Z$ on $(\mT^2)$ is given by the shift action. 
Here, $G$ is compactly generated and locally compact, but it is not distal as the shift action on $(\mT^2)^\Z$ is ergodic. 
It is easy to see that $G$ is not Lie projective. 

A locally compact group $G$ is said to be $\Lambda$-Lie projective for a subgroup $\Lambda\subset\Aut(G)$, if it admits 
compact open normal $\Lambda$-invariant subgroups $\{K_\ap\}$ such that $G/K_\ap$ is a Lie group for each $\ap$ and 
$\bigcap_\ap K_\ap=\{e\}$. Note that $\Lambda$-Lie projective groups were introduced in \cite{RS2} and they are obviously 
Lie projective. $G$ is said to be $T$-Lie projective for some $T\in\Aut(G)$ if it is $\{T^n\}_{n\in\Z}$-Lie projective. A group 
$G$ is $\Lambda$-Lie projective for a finitely generated group $\Lambda$ of $\Aut(G)$ if and only if $\Lambda\ltimes G$ 
is Lie projective, where $\Lambda$ is endowed with the discrete topology. Similarly, $G$ is $T$-Lie projective for some 
$T\in\Aut(G)$ if and only if $\Z\ltimes_T G$ is Lie projective, where the action of $n\in\Z$ on $G$ is given by the action of 
$T^n$ on $G$, and $\Z$ is endowed with the discrete topology.

We say that a locally compact group $\Lambda$ acts (continuously) on $G$ by automorphisms, if there exits a group homomorphism 
$\psi:\Lambda\to\Aut(G)$ such that the corresponding map $\Lambda\times G\to G$ given by 
$(\lambda, g)\mapsto \psi(\lambda)(g)$, $\lambda\in\Lambda$, $g\in G$, is continuous. 

For a compact group $G$, if $T\in\Aut(G)$ is distal, it follows by Lemma 2.5 of \cite{RS2}, that $G$ is $T$-Lie projective. 
As any compactly generated nilpotent group is a generalized $\ol{FC}$ group (cf.\ \cite{L2}), the following useful corollary 
follows easily from Corollary 3.7 of \cite{RS2} and Theorem \ref{distal-str}.

\begin{cor} \label{T-Lie}
Let $G$ be a compactly generated locally compact distal group. If $\Lambda$ is a compactly generated 
nilpotent group which acts distally on $G$ by automorphisms, then $G$ is $\Lambda$-Lie projective. In particular, 
if $T\in\Aut(G)$ acts distally on $G$, then  $G$ is $T$-Lie projective; i.e.\ $\Z\ltimes_T G$ is Lie projective.
\end{cor}

The next corollary will be useful, it is well-known and it can be easily deduced from Theorem 2 of \cite{L1} and 
Lemma 3.1 of \cite{Da}. Since nilpotent groups are distal, one can also use Theorem \ref{distal-str} instead of Theorem 2 of \cite{L1}
to prove it.

\begin{cor} \label{nilp-str} Any locally compact compactly generated nilpotent group 
admits a unique maximal compact subgroup. 
\end{cor}

The following group theoretic result, which may be known, will be useful in proving Theorem \ref{distal-gp}. It 
implies for a special case of compactly generated $G$ below that the unique maximal compact subgroup 
centralises $G^0$. In particular it implies that $G^0$ is central in any compact nilpotent group $G$. 

\begin{prop} \label{centraliser-nilp} Let $G$ be a locally compact nilpotent group. Then any compact subgroup of $G$ 
centralises $G^0$. 
\end{prop}

\begin{proof} Let $K$ be a compact subgroup of $G$. Since $G^0$ is normal and $K$ is compact, we get that $KG^0$ is 
a closed subgroup. Also, it is compactly generated since $G^0$ is so. By Corollary \ref{nilp-str}, $KG^0$ has a unique 
maximal compact subgroup. To prove the assertion, we may replace $G$ by $KG^0$ and also assume that $K$ is the 
unique maximal compact subgroup of $G$, and show that it centralises $G^0$. 

As $K$ is a unique maximal compact subgroup, it is characteristic in $G$, and hence it is normal in $G$. For $g\in G$, 
let $\inn(g)$ denote the inner automorphism of $G$ by the element $g$, i.e.\ $\inn(g)(x)=gxg^{-1}$, $x\in G$. 
As $G^0$ is connected, $\inn (g)|_K\in [\Aut(K)]^0$ 
for all $g\in G^0$. By Theorem $1'$ of \cite{Iw}, $[\Aut(K)]^0 = [\Inn(K)]^0 = \{\inn(k) \mid k \in K^0\}$. That is, given 
$g\in G^0$, there exists $k\in K^0$ such that $\inn(g)|_K=\inn (k)|_K$. To show that $G^0$ centralises $K$, it is enough 
to show that $K^0$ is central in $K$. Now we may assume that $G$ is a compact nilpotent group and show that $G^0$ is 
central in $G$. Since $G$ is compact, it is Lie projective, hence it is enough to prove this for a compact nilpotent Lie group $G$. 
As $G^0$ is compact and nilpotent, it is abelian (cf.\ \cite{Iw}, Lemma 2.2). Moreover, $G/G^0$ is finite. 

We prove by induction on the length $l(G)$ of the central series of the compact nilpotent Lie group $G$ that $G^0$ is central 
in $G$. If $l(G)=1$, then $G$ is abelian. Suppose for some $k\in \N$, the above statement holds for all such $G$ with  
$l(G)\leq k$. Now let $G$ be such that $l(G)=k+1$. Let $Z(G)$ be the center of $G$. Then $l(G/Z(G))=k$. By the induction hypothesis, 
$(G/Z(G))^0$ is central in $G/Z(G)$. Let $x\in G$, $g\in G^0$ and let $x_g=xgx^{-1}g^{-1}$. Since $(G^0Z(G))/Z(G)=(G/Z(G))^0$, 
from the preceding assertion we have that $x_g\in Z(G)$. Since $G/G^0$ is finite, we get that $x^n\in G^0$ for some $n\in\N$. 
Since $x_g\in Z(G)$, we get that  $x_g^n=xg^nx^{-1}g^{-n}=x^ngx^{-n}g^{-1}=e$ as $G^0$ is abelian. Therefore, 
$x$ centralises $g^n$. Since this holds for all $g\in G^0$, which is connected and abelian, we get that $x$ centralises $G^0$. 
Since this is true for all $x\in G$, we have that $G^0$ is central in $G$. Now the proof is complete by induction. 
\end{proof}

\section{Distal Actions of Automorphisms on Sub$_{\boldsymbol G}$ for Discrete groups ${\boldsymbol G}$ and Applications to Lattices} 

Let $G$ be a locally compact (metrizable) group. A sub-basis of the Chabauty topology on $\Sub_G$ is given by the sets 
$\mathcal{O}_1(K)=\{A\in \Sub_G\mid A\cap K=\emptyset\}$, $\mathcal{O}_2(U)=\{A\in \Sub_G\mid A\cap U\neq\emptyset\}$, 
where $K$ is a compact and $U$ is an open subset of $G$. As observed earlier, $\Sub_G$ is compact and metrizable.  
For details on the Chabauty topology see \cite{BHK} and \cite{PH}.

We first state a criterion for convergence of sequences in $\Sub_G$ (cf.\ \cite {BP}). 

\begin{lem} \label{conv} Let $G$ be a locally compact first countable $($metrizable$)$ group. 
A sequence $\{H_n\}$ converges to $H$ in $\Sub_G$ if and only if the following hold:
\begin{enumerate}
\item[$(i)$] For any $h\in H$, there exists a sequence $\{h_n\}$ with $h_n\in H_n$, $n\in\N$, such that $h_n\to h$.
\item[$(ii)$] For any unbounded sequence $\{n_k\}\subset\N$, if $\{h_{n_k}\}_{k\in\N}$ is such that 
$h_{n_k}\in H_{n_k}$, $k\in\N$, and $h_{n_k}\to h$, then $h\in H$.
\end{enumerate}
\end{lem}
 
 We define $\Sub_G^c$ as the space of all  closed cyclic subgroups of $G$. In general, $\Sub_G^c$ need not be 
 closed in $\Sub_G$, e.g.\ $\Sub_\R^c$ is dense in $\Sub_\R$ and $\Sub_\R=\Sub_\R^c\cup\{\R\}$. We will show 
 that for a certain class of groups $G$, which include discrete finitely generated nilpotent groups, $\Sub^c_G$ is closed. 
 We first state a useful lemma about the limits of sequences in a discrete group. We give a short proof for the sake of 
 completeness. 
 
 \begin{lem} \label{discr-conv}
 Let $G$ be a discrete group. If $\{H_n\} \subset \Sub_G$ is such that $H_n\to H$ in $\Sub_G$, then $H=\liminf H_n$.
 \end{lem}
 
 \begin{proof}
 By Lemma \ref{conv}\,$(i)$, we get that for every $h\in H$, there exist $h_n\in H_n$, $n\in\N$, such that 
 $h_n\to h$. Since $G$ is discrete we get that $h_n=h$ for all large $n$. Therefore, $h\in \liminf H_n$, 
 and hence $H \subset\liminf H_n$. Conversely, suppose $h\in\liminf H_n$. Then there exists $k\in\N$ 
 such that $h\in H_n$, for all $n\geq k$. Now by Lemma \ref{conv}\,$(ii)$, we get that $h\in H$. Hence 
 $$H=\bigcup_{k=1}^{\infty}\bigcap_{n=k}^{\infty} H_n=\liminf H_n.$$
 \end{proof}
 
 As any discrete subgroup of a connected solvable Lie group is finitely generated (cf.\ \cite{Rag}, Corollary 3.9), 
 the following will be useful for results on lattices of a connected solvable Lie group 
 (see Theorem \ref{lattice-solv}).
 
 \begin{lem} \label{fin-gen}
Let $G$ be a discrete finitely generated group such that all its subgroups are also finitely generated. Then $\Sub_G^c$ is closed.
\end{lem}

\begin{proof} Let $\{H_{n}\}$ be a sequence in $\Sub_G^c$ such that $H_n\to H$. By Lemma \ref{discr-conv}, 
$H=\bigcup_{k=1}^{\infty}G_k$, where $G_k=\bigcap_{n=k}^{\infty} H_n$. From the hypothesis, $H$ is finitely generated. 
Let $\{x_1,\hdots, x_m\}$ be the set of generators for $H$. Since $H$ is an increasing union of cyclic groups $G_k$, 
there exists $n_0\in \N$ such that $x_1,\hdots, x_m\in G_{n_0}$. Therefore $H=G_{n_0}$, and hence $H$ is cyclic. 
\end{proof}

 A locally compact group $G$ is said to be {\em strongly root compact} if for every compact subset $C$ of $G$, there exists 
 a compact subset $C_0$ of $G$ with the property that for every $n\in\N$, the finite sequences $\{x_1,\ldots, x_n\}$ of $G$ 
 with $x_n=e$, satisfying  $Cx_iCx_j\cap Cx_{i+j}\ne\emptyset$ for all $i+j\leq n$, are contained in $C_0$ 
 (see \cite{He}, Definition 3.1.10). All compact groups and compactly generated nilpotent groups are strongly root compact 
 (cf.\ \cite{He}, Theorem 3.1.17). 
 
 For any $g\in G$, let $R_g=\{x\in G\mid x^n=g \mbox{ for some } n\in\N\}$, the set of roots of $g$ in $G$. If $G$ is strongly root compact,
 then by Theorem 3.1.13 of \cite{He}, $R_g$ is relatively compact for every $g\in G$.
 
\begin{lem} \label{cyc-closed} Let $G$ be a discrete group.
If for every $g\in G$, the set $R_g$ of roots of $g$ is finite, then $\Sub_G^c$ is closed.
In particular, if $G$ is a strongly root compact, then $\Sub_G^c$ is closed. 
\end{lem}

\begin{proof} For a discrete group $G$, suppose $R_g$ is finite for every $g\in G$. 
 Let $\{H_n\}$ be a sequence in $\Sub_G^c$ such that $H_n\to H$. By Lemma \ref{discr-conv} we get $H=\liminf H_n$. 
 Let $G_k=\bigcap_{n=k}^{\infty}H_n$, $k\in \N$. Then for all $k\in\N$, $G_k\subset G_{k+1}\subset H_{k+1}$, 
 $G_k\subset H$ and each $G_k$ is cyclic. If $H=\{e\}$, then there is nothing to prove. Suppose $G_k$ is finite for all $k\in\N$. 
Since $H=\bigcup_{k=1}^{\infty}G_k$, it consists of finite order elements. Hence $H\subset R_e$ which is finite, so $H$ is finite. 
Therefore, $H=G_k$ for some $k$, and hence $H$ is cyclic. Now suppose there exists $m\in\N$ such that $G_m$ is infinite. 
Replacing $\{H_n\}$ by $\{H_{n+m}\}$, we may assume that $G_1$ is infinite, and hence that $G_k$ is an infinite cyclic group, $k\in\N$.

Let $x_k$ be a generator of $G_k$, $k\in\N$.  As $G_k\subset G_{k+1}$, $k\in\N$, replacing $x_{k+1}$ by its inverse if necessary, 
we get that there exists $l_k\in\N$ such that $x_k=x_{k+1}^{l_k}$. Hence $x_1=x_k^{n_k}$ for all $k\geq 2$, 
where $n_k=l_1\cdots l_{k-1}$. From the hypothesis, $R_{x_1}$ is finite, and hence 
$\{x_k\}_{k\in\N}$ is finite. As each $x_k$ generates an infinite cyclic group $G_k$ and $G_k\subset G_{k+1}$, $k\in\N$, 
there exists $n_0\in\N$ such that $G_k=G_{k+1}$ for all $k\ge n_0$, and hence 
$H=G_{n_0}=\bigcap_{n=n_0}^{\infty}H_n$. Therefore, $H$ is cyclic. This proves
that $\Sub_G^c$ is closed. 

Now suppose $G$ is a strongly root compact discrete group. By Theorem 3.1.13 of \cite{He}, $R_g$ is relatively compact, and 
hence, finite for every $g\in G$. Now it follows from the first statement that $\Sub_G^c$ is closed. 
\end{proof}

The class $\NC$ of automorphisms is defined in \cite{SY3}. An automorphism $T$ of a locally compact metrizable group $G$ 
belongs to class $\NC$ if given any closed cyclic subgroup $A$ of $G$, $T^{n_k}(A)\not\to\{e\}$ for any unbounded sequence 
$\{n_k\}\subset\Z$. For $x\in G$, let $G_x$ denote the cyclic group generated by $x$ in $G$. Then either $G_x$ is closed 
(and hence discrete) or $\ol{G_x}$ is compact. 

Note that if $T$ acts distally on $\Sub^a_G$, then $T\in\NC$. It is easy to see that $T^n\in\NC$ for some $n\in\Z\setminus\{0\}$ 
if and only if $T^n\in\NC$ for all $n\in\Z$. We now state and prove an elementary result about the class $\NC$ for the 
discrete quotient groups.

\begin{lem} \label{discr-quo} Let $G$ be a locally compact first countable $($metrizable$)$ group, $T\in\Aut(G)$ and let $H$ be an 
open normal $T$-invariant subgroup of $G$. Let $\T:G/H\to G/H$ be the automorphism of $G/H$ corresponding to $T$. For any 
$x\in G$, let $\bar x$ denote the element $xH\in G/H$. Suppose $T\in\NC$. If $x\in G\setminus H$ is such that $G_{\bar x}$ is 
infinite, then $\T^{n_k}(G_{\bar x})\not\to\{\bar e\}$, for any unbounded sequence $\{n_k\}\subset\Z$. In particular, if $G/H$ is 
torsion-free, then $\T\in\NC$. 
\end{lem}

\begin{proof} As $H$ is open in $G$, it is also closed and $G/H$ is a discrete group. Let $x\in G$ be such that $G_{\bar x}$ 
is infinite. If possible, suppose there exists an unbounded sequence $\{n_k\}\subset\Z$ such that $T^{n_k}(G_{\bar x})\to\{\bar e\}$. 
Let $G_x$ be the group generated by $x\in G$. 
Since $G_{\bar x}$ is closed, discrete and infinite, we have that $G_x$ is also closed, discrete and infinite. As $\Sub_G$ is compact, 
we may choose a subsequence of $\{n_k\}$, and denote it by $\{n_k\}$ again, such that $T^{n_k}(G_x)\to L$ for some $L\in\Sub_G$. 
Since $T\in\NC$, $L\ne\{e\}$. Let $g\in L\setminus\{e\}$, then $T^{n_k}(x^{m_k})\to g$, for some $\{m_k\}\subset\Z$. Now 
$T^{n_k}({\bar x}^{m_k})\to \bar g$. As $T^{n_k}(G_{\bar x})\to\{\bar e\}$, we have that $g\in H$, and hence 
$T^{n_k}({\bar x}^{m_k})=\{\bar e\}$ for all large $k$, since $G/H$ is discrete. This implies that $\bar x$ has finite order, 
which leads to a contradiction. Therefore, $\T^{n_k}(G_{\bar x})\not\to\{\bar e\}$. If $G/H$ is torsion-free, every nontrivial 
element of  $G/H$ generates a discrete infinite group, and hence the last assertion follows easily. 
\end{proof}

For a locally compact group $G$ and $T\in\Aut(G)$, let $M(T)=\{x\in G\mid \{T^n(x)\}_{n\in\Z} \mbox{ is relatively compact}\}$. 
It is a $T$-invariant subgroup of $G$. The following basic lemma about automorphisms of strongly root compact groups in the 
class $\NC$ will be very useful. 

\begin{lem} \label{mt}
Let $G$ be a locally compact  first countable $($metrizable$)$ strongly root compact group. Let $T\in\Aut(G)$ be such that $T\in\NC$.
Then $\{x\in G\mid G_x \mbox{ is closed}\}${ }$\subset M(T)$. 
\end{lem}

\begin{proof} 
Let $x\in G$ be such that $G_x$ is closed and let $O_x=\{T^n(x)\}_{n\in\Z}$. Since $G$ is locally compact and metrizable, and 
$\ol{O_x}$ is separable, it is second countable, and hence $\ol{O_x}\subset \bigcup_{n\in\N} V_n$, for some open relatively 
compact sets $V_n$, and we may also assume that $V_n\subset V_{n+1}$ for all $n\in\N$. For some $n\in\N$, if 
$O_x\subset V_n$, then $\ol{O_x}\subset \ol{V_n}$ and hence $\ol{O_x}$ is compact.
 
If possible, suppose $O_x\not\subset V_n$, $n\in\N$. There exists $k_n\in\Z$ such that $|k_n|\geq n$ and 
$T^{k_n}(x)\not\in V_n$, $n\in\N$.  As $G_x\in\Sub_G$ and the latter is compact, passing to a subsequence
if necessary, we get that $T^{k_n}(G_x)\to H$ in $\Sub_G$ for some closed subgroup $H$ in $G$. Since $T\in\NC$, 
it implies that $H\ne\{e\}$. Let $a\in H$ be such that $a\ne e$. By Lemma \ref{conv}\,$(i)$, there exists a sequence 
$\{m_n\}\subset\Z$ such that $T^{k_n}(x^{m_n})\to a$, and hence $\{T^{k_n}(x^{m_n})\}$ is relatively compact. 
Replacing $a$ by $a^{-1}$ if necessary, and passing to a subsequence, we may assume that $\{m_n\}\subset\N$. 
As $G$ is strongly root compact, by Theorem 3.1.13 of \cite{He}, we get that $\{T^{k_n}(x)\}$ is relatively compact 
and hence, it has a limit point (say), $b$. Then $b\in\ol{O_x}$. As $V_n$ is increasing, we get that for every $m\in\N$, 
$\{T^{k_n}(x)\}_{n\geq m}\subset G\setminus V_m$ which is closed. It follows that $b\not\in V_m$ for every $m\in\N$, and
hence $b\not\in\bigcup_{m\in\N} V_m$; this leads to a contradiction since $b\in \ol{O_x}$. 
Therefore, $O_x\subset V_n$ for some $n\in\N$ and hence $\ol{O_x}$ is 
compact and $x\in M(T)$. This proves the assertion. 
\end{proof}

 In a discrete group every element generates a discrete (closed) cyclic group. The next corollary follows easily from the proof of Lemma \ref{mt} 
 as a sequence in a discrete group converges if and only if it is eventually constant and, discrete compact sets are finite. 

\begin{cor} \label{trivia} Let $G$ be a discrete group such that the set $R_g$ of roots of $g$ is finite for every $g\in G$. 
Let $T\in \Aut(G)$ be such that $T\in\NC$. Then $G=M(T)$. That is, for every $x\in G$, the $T$-orbit of $x$, 
$\{T^n(x)\}_{n\in\Z}$ is finite.
\end{cor}

For any discrete group, all automorphisms are distal. For strongly root compact discrete groups $G$, or more 
generally for discrete groups $G$ in which the set of roots of every element is finite, the following proposition shows that 
only finite order automorphisms of $G$ act distally on $\Sub_G$. Note that for such groups $G$, 
$\Sub^c_G$ is closed by Lemma \ref{cyc-closed}. The proposition holds in particular for discrete finitely generated 
nilpotent groups as they are strongly root compact (cf.\ \cite{He}). 

\begin{prop} \label{discr-subg}
Let $G$ be a discrete finitely generated group and let $T\in\Aut(G)$. 
Suppose the set $R_g$ of roots of $g$ is finite for every $g\in G$. 
Then the following are equivalent:
\begin{enumerate}
\item[{$1.$}] $T\in\NC$
\item[{$2.$}] $T$ acts distally on $\Sub_G^c$.
\item[{$3.$}] $T$ acts distally on $\Sub_G$.
\item[{$4.$}] $T^n=\Id$, where $\Id$ is the identity map.
\end{enumerate}
In particular, if $G$ is strongly root compact, then $(1-4)$ above are are equivalent. 
\end{prop}

\begin{proof}
$4\implies 3 \implies 2\implies 1$ is obvious. It is enough to show that $1\implies 4$. Suppose $T\in\NC$. 
Then by Corollary \ref{trivia}, $G=M(T)$. Now as $G$ is discrete, for every $x\in G$, the $T$-orbit of $x$, 
$\{T^n(x)\}_{n\in\Z}$ is finite, and hence $T^m(x)=x$ for some $m\in\N$. As $G$ is finitely generated, 
there exist $x_1,\ldots, x_l\in G$ which generate $G$. Let $n_1,\ldots, n_l\in\N$ be such that $T^{n_i}(x_i)=x_i$, 
$1\leq i\leq l$. Let $n={\rm lcm}(n_1,\ldots, n_l)$. Then $T^n(x)=x$ for all $x\in G$, i.e.\ $T^n=\Id$. 

If $G$ is strongly root compact, then by Theorem 3.1.13 of \cite{He}, $R_g$ is finite for every $g\in G$, and $(1-4)$ are equivalent
from above. 
\end{proof}

Proposition \ref{discr-subg}, in particular, implies that if $T\in\GL(n,\Z)$, ($n\geq 2$), does not have finite order, then $T$ does not act 
distally on $\Sub_{\Z^n}$; e.g.\ $T$ is any nontrivial strictly upper triangular matrix in $\GL(n,\Z)$ with all its diagonal entries equal to 1.
In fact, since $GL(n,\Z)$ is virtually torsion-free by Selberg's Lemma, there exists a subgroup (say) $L$ of finite index in $GL(n,\Z)$ 
which is torsion-free, and hence every nontrivial $T\in L$ does not acts distally on $\Sub_{\Z^n}$. 
As an application of the proposition, we get the following corollary which relates the behaviour of an automorphism 
of a lattice $\Gamma$ in a simply connected nilpotent group $G$ in terms of the distality of its action on $\Sub_\Gamma$ and on 
$\Sub_G$. Note that any automorphism of such a $\Gamma$ extends to a unique automorphism of $G$ (cf.\ \cite{Rag}). 
Note also that such a $\Gamma$ is finitely generated and strongly root compact, and hence by Lemma \ref{cyc-closed}, 
$\Sub^c_\Gamma$ is closed. 

\begin{cor} \label{lattice-nilp}
Let $G$ be a connected simply connected nilpotent Lie group and let $\Gamma$ be a lattice in $G$. Let $T\in \Aut(G)$ 
be such that $T(\Gamma)=\Gamma$. Then the following are equivalent:
\begin{enumerate}
\item[{$(1)$}] $T|_\Gamma\in\NC$. 
\item[{$(2)$}] $T$ acts distally on $\Sub_\Gamma^a$.
\item[{$(3)$}] $T$ acts distally on $\Sub_\Gamma$.
\item[{$(4)$}] $T\in\NC$.
\item[{$(5)$}] $T$ acts distally on $\Sub_G^a$. 
\item[{$(6)$}] $T$ acts distally on $\Sub_G$. 
\item[{$(7)$}] $T$ acts distally on $\Sub_\Gamma^c$.
\item[{$(8)$}] $T^n=\Id$, where $\Id$ is the identity map. 
\end{enumerate}
\end{cor}

\begin{proof}
$(8)\implies (6)\implies (5) \implies (4) \implies (1)$ and $(6)\implies (3)\implies (2) \implies (7) \implies (1)$ are obvious. It is enough 
to show that $(1)\implies (8)$. Since $\Gamma$ is a lattice in a simply connected nilpotent group $G$, it is finitely generated, 
nilpotent and discrete. Therefore, $\Gamma$ is strongly root compact and by Proposition \ref{discr-subg}, 
$T^n|_\Gamma=(T|_\Gamma)^n=\Id$. By Theorem 2.11 of \cite{Rag} and Corollary 1 following it, 
$T^n=\Id$. 
\end{proof}

Example \ref{ex1} shows that Corollary \ref{lattice-nilp} does not hold for lattices in a general connected simply connected 
solvable Lie group and it also illustrates that the following theorem is the best possible result for lattices $\Gamma$ in a 
connected simply connected solvable Lie group. Note that such a lattice $\Gamma$ is torsion-free and every subgroup of 
it is finitely generated, and hence by Lemma \ref{fin-gen}, $\Sub^c_\Gamma$ is closed. The theorem can be viewed as a 
generalisation of Corollary \ref{lattice-nilp} as any automorphism of a lattice in a simply connected nilpotent group $G$ 
extends uniquely to that of $G$. Example \ref{ex1} also shows that not all the statements in the theorem below are equivalent. 

\begin{thm} \label{lattice-solv}
Let $G$ be a connected simply connected solvable Lie group. Let $N$ be the nilradical of $G$, the largest connected 
nilpotent normal subgroup of $G$. Suppose $G$ admits a lattice $\Gamma$ and an automorphism $T \in \Aut(G)$ 
such that $T(\Gamma)=\Gamma$. Then $(1-2)$ are equivalent as well as $(3-6)$ are equivalent.
\begin{enumerate}
\item[{$(1)$}] $T|_\Gamma\in\NC$.
\item[{$(2)$}] $T^n|_{\Gamma'}=(T|_{\Gamma'})^n=\Id$ for some $n\in\N$, where $\Gamma'$ is a normal subgroup of finite index 
in $\Gamma$ containing $\Gamma\cap N$, and $\Id$ is the identity map on $\Gamma'$. \item[{$(3)$}] $T$ acts distally 
on $\Sub^c_\Gamma$.
\item[{$(4)$}] $T$ acts distally on $\Sub^a_\Gamma$.
\item[{$(5)$}] $T$ acts distally on $\Sub_\Gamma$.
\item[{$(6)$}] $T^n|_\Gamma=(T|_\Gamma)^n=\Id$, for some $n\in\N$.
\end{enumerate}
\end{thm}

\begin{proof} Suppose (2) holds. Let $S=T^n$ and let $m$ be the index of $\Gamma'$ in $\Gamma$. Let $x\in\Gamma$. 
Then $x^m\in \Gamma'$ and $S(x^m)=x^m$ and, $x^m\ne e$ as $G$ is torsion-free. It follows that any limit point $H$ of 
$\{S^i(G_x)\}$ contains a subgroup generated by $x^m$.  Therefore, $S|_\Gamma\in\NC$, and hence
$T|_\Gamma\in\NC$ and (1) holds. 

Now suppose (1) holds, i.e.\ $T|_\Gamma\in\NC$. 
As $G$ is connected, solvable and simply connected, $[G,G]$ is a closed connected nilpotent normal subgroup. 
Also, the nilradical $N$ is simply connected, $T(N)=N$, $[G,G]\subset N$ and $G/N$ is abelian. Moreover, 
$\Gamma\cap N$ (resp.\ $(\Gamma N)/N$) is a lattice in $N$ (resp.\ in $G/N$) (cf.\ \cite{Rag}, Corollary 3.5). Since 
$T$ keeps $\Gamma\cap N$ invariant and $T$ acts distally on $\Sub^c_{\Gamma\cap N}$, by Corollary \ref{lattice-nilp},
$T^{n_1}=\Id$ on $\Gamma\cap N$ for some $n_1\in \N$. Note that $(\Gamma N)/N$ is a lattice in $G/N$ which is simply 
connected and abelian. Note also that  $\Gamma/(\Gamma\cap N)$ is isomorphic to $(\Gamma N)/N$, therefore it is 
finitely generated, abelian and torsion-free. By Lemma \ref{discr-quo}, we get that $\T\in\NC$, where 
$\T\in\Aut (\Gamma/(\Gamma\cap N))$ is the automorphism corresponding to $T$. By Proposition \ref{discr-subg}, 
$\T^{n_2}=\Id$ on $\Gamma/(\Gamma\cap N)$ for some $n_2\in \N$. Let $n={\rm lcm}(n_1,n_2)$ and let $S=T^n$. 
Then $S|_{\Gamma\cap N}=\Id$ and $S$ acts trivially on $\Gamma/(\Gamma\cap N)$. 
Now suppose $x\in\Gamma$ is such that $S(x^j)\ne x^j$ for all $j\in\N$. Then $S(x)=xy$ for some nontrivial 
$y\in\Gamma\cap N$. As $\Gamma\cap N$ is torsion-free, 
by Lemma 3.12 of \cite{SY3}, we get that $S\not\in\NC$. Hence $T\not\in\NC$, which contradicts the hypothesis. Therefore, 
for every $x\in \Gamma$, there exists $j$ which depends on $x$ such that $S(x^j)=x^j$. As $\Gamma$ is finitely generated, 
there exist $x_1,\ldots,x_k$ which are generators for $\Gamma$. Let $j_i$ be such that $S(x_i^{j_i})=x_i^{j_i}$, $1\leq i\leq k$. 
Let $m={\rm lcm}(j_1,\ldots, j_k)$ and let $\Gamma'$ be the subgroup of $\Gamma$ generated by  
$\{x_1^{j_1},\ldots, x_k^{j_k}\}\cup(\Gamma\cap N)$. Since $\Gamma/\Gamma'$ is a finitely generated abelian group 
consisting of elements of finite order, it is finite. Therefore, $\Gamma'$ is a subgroup of finite index in $\Gamma$. Also, 
$\Gamma'$ is normal in $\Gamma$, since $[\Gamma,\Gamma]\subset \Gamma\cap N\subset \Gamma'$. 
As $S|_{\Gamma\cap N}=\Id$, we get that $S|_{\Gamma'}=\Id$. As $S=T^n$, it follows that (2) holds. 

It is easy to see that $(6)\implies (5)\implies (4)\implies (3)$. It is enough to show that $(3)\implies (6)$. Suppose (3) holds, 
i.e.\ $T$ acts distally on $\Sub^c_\Gamma$. Since $(3)\implies (1)$, we have that (2) holds, i.e.\ there exists a normal subgroup 
$\Gamma'$ of finite index in $\Gamma$ containing $\Gamma\cap N$ and $n\in\N$ such that $T^n|_{\Gamma'}=\Id$, where 
$N$ is the nilradical of $G$. As in the proof above, we can choose $\Gamma'$ and $n$ such that $T^n$ acts trivially on 
$\Gamma/(\Gamma\cap N)$, 

Let $S=T^n$. Then 
we show that $S|_\Gamma=\Id$. If possible, suppose $x\in\Gamma$ is such that $S(x)\ne x$. As $x^m\in\Gamma'$, for some 
$m\in\N$, we get that $S(x^m)=x^m$. Let $k$ be the smallest positive integer such that $S(x^k)=x^k$. Then $k\geq 2$. Now 
$S(x^l)=x^ly_l$ for some $y_l\in\Gamma\cap N$, $y_l\ne e$ and $S^i(x^l)=x^ly_l^i$ for all for $1\le l\le k-1$, $i\in\N$. Since 
$G$ is torsion-free. $\{S^{i_j}(x^l)\}$ has no limit point if $i_j\to\infty$ and $1\le l\le k-1$. Now it is easy to show that 
$S^i(G_x)\to G_{x^k}$ in $\Sub^c_\Gamma$, as $i\to\infty$, where $G_x$ (resp.\ $G_{x^k}$) is the cyclic group 
generated by $x$ (resp.\ $x^k$) in $\Gamma$. As $S(G_{x^k})=G_{x^k}$ and $k\geq 2$, it implies that $S$ does not 
act distally on $\Sub^c_\Gamma$. Since $S=T^n$, we get that $T$ does not act distally on $\Sub^c_\Gamma$. 
This contradicts (3). Therefore, $S|_\Gamma=\Id$, and hence (6) holds.
 \end{proof}

The following is an example of a connected simply connected solvable Lie group $G$ which admits a nontrivial automorphism 
$T$ and a lattice $\Gamma_1$ such that $T|_{\Gamma_1}=\Id$ and $T\not\in\NC$. This is unlike the case of simply connected 
nilpotent groups (see Corollary \ref{lattice-nilp}). The example also shows that there exists a lattice $\Gamma_2$ in $G$ 
such that $T|_{\Gamma_2}\in\NC$ but it does not act distally on $\Sub^c_{\Gamma_2}$. 

\begin{example} \label{ex1}
Let $G=\R\ltimes\R^2$ where the group operation is given by $(s,x)(t,y)=(s+t, e^{2i\pi t}x+y)$, $s,t\in \R$ and $x,y\in\R^2$. 
Then $G$ is a connected simply connected solvable Lie group. Let $T$ be an inner automorphism by some 
$g\in\Z^2\setminus\{0\}$, i.e.\ $T(t,y)=(t, y+e^{2i\pi t}g-g)$, for all $(t,y)$ as above. Let $\Gamma_1=\Z\times\Z^2$, 
where $\Z$ is a lattice in $\R$ and $\Z^2$ is a lattice in the normal subgroup $\R^2$. Then $\Gamma_1$ is a lattice in $G$ 
and $T|_{\Gamma_1}=\Id$. Also, $T|_{\R^2}=\Id$ and the action on $G/\R^2$ corresponding to $T$ is also trivial. Now choose 
an irrational number $t$ in $\R$. Then $T(t)=(t,e^{2i\pi t}g-g)$, and hence $T(mt)\ne mt$ for all $m\in\Z$, i.e.\ $T$ does not fix any 
nontrivial element in the discrete cyclic group $G_t$ generated by $t$ in $\R$. As $\R^2$ has no nontrivial compact subgroup, 
it is easy to show that $T^n(G_t)\to\{(0,0)\}$ in $\Sub_G$ as $n\to\infty$. Therefore, $T\not\in\NC$ 
(this also follows from Lemma 3.12 of \cite{SY3}). 

Now choose $\Gamma_2=\frac{1}{2}\Z\ltimes \Z^2$ and $T$ is the inner automorphism by $g$ as above, where 
$g\in\Gamma_1\cap \Z^2\setminus\{0\}$. 
Then $\Gamma_2$ is a $T$-invariant lattice in $G$, $\Gamma_1\subset\Gamma_2$ and $T|_{\Gamma_1}=\Id$. 
For any $x\in\Gamma_2$, $x^2\in\Gamma_1$. Therefore, it is easy to see that $T|_{\Gamma_2}\in\NC$. For 
$t=\frac{1}{2}\in\Gamma_2\cap\R$, $T(t)=(t,-2g)$. As $g\ne 0$, it is easy to check that 
$T^n(G_t)\to \Z=\Gamma_1\cap\R$ as $n\to\infty$, where $G_t$ is the cyclic group generated by $t$ in $\Gamma_2$. 
As $\Gamma_1\cap\R$ is cyclic and $T(\Gamma_1\cap\R)=\Gamma_1\cap\R\ne G_t$, $T$ does not act distally on 
$\Sub^c_{\Gamma_2}$.
\end{example}

Now we study the action of automorphisms of a lattice $\Gamma$ in a connected semisimple Lie group on 
$\Sub_\Gamma$. We first give an example of an automorphism $T$ of $\SL(2,\Z)$, which does not belong to $\NC$, and 
hence it does not act distally on $\Sub^a_{SL(2,\Z)}$.

\begin{example} \label{ex2}
 Let $T=\inn(g)$, the inner automorphism of $\SL(2,\Z)$, where
 $$g=\begin{bmatrix}
1&1\\
0&1\\
\end{bmatrix}. 
\ \ \ \mbox{For} \ \ \ x=\begin{bmatrix}
1&0\\
1&1\\
\end{bmatrix}$$ 
and, for $\{n_k\}$ and $\{l_k\}$ in $\Z$,
$$T^{n_k}(x^{l_k})=g^{n_k}x^{l_k}g^{-n_k}=
\begin{bmatrix}
  1+n_kl_k & -n_k^2l_k\\
     l_k     &  1-n_kl_k\\
\end{bmatrix}. 
$$
If $n_k\to \infty$ and $l_k\ne 0$, at least one of the entries of $T^{n_k}(x^{l_k})$ goes to $\infty$, and hence 
$\{T^{n_k}(x^{l_k})\}$ does not converge in $SL(2,\Z)$. This implies that for the cyclic group $G_x$ generated
by $x$ in $\SL(2,\Z)$,  since $\{T^{n_k}(G_x)\}$ converges for some unbounded sequence $\{n_k\}\subset\N$, 
we have that $T^{n_k}(G_x)\to\{e\}$. Therefore, $T\not\in\NC$, and hence $T$ does not act distally on $\Sub^a_{\SL(2,\Z)}$.
\end{example}

The question arises for a lattice $\Gamma$ in $G$, if the action of automorphisms of $\Gamma$ on $\Sub_\Gamma$ for 
a connected semisimple group $G$ behave in the same way or differently from the case when $G$ is a simply connected 
nilpotent Lie group. Theorem \ref{newp} shows that an almost similar result as above hold in this case too. 

We first state and prove some lemmas to use later.  The following lemma may be known but we give a proof for 
the sake of completeness. An element $g\in\GL(n,\R)$ is said to be {\it net} if the multiplicative group generated by 
the eigenvalues of $g$ in $\C\setminus\{0\}$ is torsion-free. Note that $g$ is net if and only if $g_s$ is net, where 
$g_s$ is the semisimple part of $g$ in its multiplicative Jordan decomposition. A subgroup of $\GL(n,\R)$ is said to 
be {\it net} if all its elements are net (see 17.1 in \cite{B}). 

\begin{lem} \label{newl} Let $G$ be a connected semisimple Lie group and let $\Gamma$ be a lattice in $G$. Then there 
exists a normal subgroup $\Gamma'$ of finite index in $\Gamma$ such that 
$R_g:=\{x\in \Gamma' \mid x^n=g\mbox{ for some } n\in\Z\}$ is finite for all $g\in\Gamma'$. Moreover, the torsion elements 
in $\Gamma'$ form a finite central subgroup in $G$. 
\end{lem}

\begin{proof} For the center $Z(G)$ of $G$, $G/Z(G)$ is a linear subgroup 
of $\GL(n,\R)$, for some $n\in\N$. Let $\pi:G\to G/Z(G)$ be the natural projection. Since $\Gamma$ is a lattice in $G$, it is 
finitely generated, Hence $\pi(\Gamma)$ is finitely generated, and by Corollary 17.7 of \cite{B}, $\pi(\Gamma)$ has a subgroup 
of finite index (say) $\Gamma_1$ such that it is net. Let $\Gamma'=\pi^{-1}(\Gamma_1)\cap \Gamma$. Then $\Gamma'$ is 
a subgroup of finite index in $\Gamma$. 

For $x\in \Gamma$, let $\bar x=\pi(x)$ and let $G_x$ (resp.\ $G_{\bar x}$) be the cyclic group generated by $x$ in 
$\Gamma$ (resp.\ $\bar x$ in $\pi(\Gamma)$). Then for any $x\in\Gamma'$, the Zariski closure $\tilde G_{\bar x}$ of 
$G_{\bar x}$ is connected (see the proof of Proposition 17.2 in \cite{B}). 

Let $g\in \Gamma'$. First suppose that $g\in Z(G)$. Then $\pi(R_g)$ is the group of torsion elements in $\Gamma_1$. 
Since $\Gamma_1$ is net, it is torsion-free, and hence $\pi(R_g)=\{\bar e\}$. Therefore, $R_g\subset Z(G)$. As $Z(G)$ 
is compactly generated and abelian, $R_g$ is finite. This implies in particular that the set of torsion elements in $\Gamma'$
is a subgroup of $Z(G)$, and hence it is finite since $Z(G)$ is compactly generated and abelian. 

Now suppose $g\not\in Z(G)$ and let $x\in R_g$. Then $\bar g,\bar x\in\Gamma_1$ are nontrivial and $G_{\bar g}$ is 
a subgroup of finite index in $G_{\bar x}$, and since each of them have connected Zariski closure, we get that 
$\tilde G_{\bar g}=\tilde G_{\bar x}$. That is, $\bar x\in\tilde G_{\bar g}$. Let $H_g=\pi^{-1}(\Gamma_1\cap \tilde G_{\bar g})$. 
Since $\tilde G_{\bar g}$ is connected and abelian and since $Z(G)$ is finitely generated and central in $G$, we get that 
$H_g$ is a finitely generated nilpotent group. From above, we have that $R_g\subset H_g$. Since $H_g$ is finitely generated 
and nilpotent, it is strongly root compact and by Theorem 3.1.13 of \cite{He}, $R_g$ is finite. 

Replacing $\Gamma'$ by a smaller subgroup of finite index if necessary, we may assume that $\Gamma'$ is normal. 
\end{proof}

Now we can deduce the following: 

\begin{lem} \label{new2} Let $\Gamma$ be a lattice in a connected semisimple Lie group $G$. Then $\Sub^c_\Gamma$ is closed. 
\end{lem}

\begin{proof} By Lemma \ref{newl}, there exists a normal subgroup $\Gamma'$ of finite index in $\Gamma$ 
such that for every $g\in\Gamma'$, the set $R_g$ of roots of $g$ in $\Gamma'$ is finite. By Lemma \ref{cyc-closed}, we get that 
$\Sub^c_{\Gamma'}$ is closed. If $\Gamma=\Gamma'$, then $\Sub^c_\Gamma=\Sub^c_{\Gamma'}$ and it is closed. 
Now suppose $\Gamma'$ is proper subgroup of $\Gamma$. 

Let $\{H_n\}_{n\in\N}$ be a sequence in $\Sub^c_\Gamma$ such that $H_n\to H$ in $\Sub_\Gamma$. We need to show that 
$H$ is cyclic. Let $H_n'=H_n\cap \Gamma'$, $n\in\N$ and let $H'=H\cap\Gamma'$. Then $H'_n\in \Sub^c_{\Gamma'}$ and it 
follows by Lemma \ref{discr-conv} that $H'_n\to H'$. From above we have that $H'$ is cyclic. Since $H\Gamma'/\Gamma'$ is finite 
and isomorphic to $H/(H\cap\Gamma')$. we have that $H'=H\cap\Gamma'$ is a normal subgroup of finite index in $H$, and hence 
$H$ is finitely generated. Let $\{h_1,\ldots, h_m\}$ be a set of generators in $H$. Then there exists $k\in\N$ such that for 
$1\leq i\leq m$, $h_i\in \bigcap_{n=k}^\infty H_n\subset H_k$. Therefore, $H\subset H_k$ and hence $H$ is cyclic. 
\end{proof}

The following lemma should be known; we give a short proof for the sake of completeness. Recall that for a group $H$, $Z(H)$ 
denotes the center of $H$. 

\begin{lem} \label{newl-center} Let $G$ be a connected semisimple Lie group, $\Gamma$ be a lattice in $G$ and let $\Gamma'$ 
be a subgroup of finite index in $\Gamma$. Then $Z(\Gamma)\cap\Gamma'$ is a subgroup of finite index in $Z(\Gamma')$. 
\end{lem}

\begin{proof} 
First suppose that $G$ has no compact factors i.e.\ the maximal compact connected normal subgroup of $G$ is 
trivial. By Corollary 5.18 of \cite{Rag}, $Z(\Gamma')\subset Z(G)$, the center of $G$. Hence $Z(\Gamma')\subset Z(\Gamma)$. 
Now suppose $G$ has a nontrivial compact factor. Let $K$ the largest compact connected 
normal subgroup of $G$. If $G=K$, then $\Gamma$ and $\Gamma'$ are finite and the assertion follows trivially. Suppose $G$ is 
not compact. Then $G/K$ is semisimple and it has no compact factors. Let $\psi:G\to G/K$ be the natural projection. Then 
$\psi(\Gamma)$ and $\psi(\Gamma')$ are lattices in $G/K$. From above, we have that 
$\psi(Z(\Gamma'))\subset Z(\psi(\Gamma'))\subset Z(G/K)$. Therefore, $xgx^{-1}g^{-1}\in K$ for all $g\in G$ and 
$x\in Z(\Gamma')$, Fix $x\in Z(\Gamma')$ and let $g\in\Gamma$. 

We first assume that $\Gamma'$ is normal in $\Gamma$. As $Z(\Gamma')$ is normal in $\Gamma$, 
$xgx^{-1}g^{-1}\in Z(\Gamma')\cap K$; this is a finite abelian group. Let $m$ be the order of $Z(\Gamma')\cap K$. 
Since $Z(\Gamma')$ is abelian, we get that $(xgx^{-1}g^{-1})^m=x^mgx^{-m}g^{-1}=e$. Therefore, $x^m\in Z(\Gamma)$. 

Note that the center of any connected semisimple Lie group is compactly generated. Therefore, 
the center of any lattice in $G$ is compactly generated as its image in $G/K$ is central in $G/K$, where $K$ as above is compact. 
Since $Z(\Gamma')$ is compactly generated and abelian and $x^m\in Z(\Gamma)$ for every $x\in Z(\Gamma')$, we have that 
$Z(\Gamma')/(Z(\Gamma)\cap \Gamma')$ is finite. 

Now suppose $\Gamma'$ is not normal in $\Gamma$, there exists a normal subgroup $\Gamma''$ of finite index in $\Gamma$ 
such that $\Gamma''\subset \Gamma'$. From the above discussion, we have that $Z(\Gamma'')/(Z(\Gamma)\cap\Gamma'')$ 
and $Z(\Gamma'')/(Z(\Gamma')\cap \Gamma'')$ are finite. Since $Z(\Gamma')/(Z(\Gamma')\cap \Gamma'')$ is  also finite, it is 
easy to deduce that $Z(\Gamma')/(Z(\Gamma)\cap \Gamma')$ is finite. 
\end{proof}

Using the above lemmas we can prove the following result which was suggested by an anonymous referee along with a sketch of a proof. 

\begin{thm} \label{newp} Let $G$ be a connected semisimple Lie group and let $\Gamma$ be a lattice in $G$. 
Let $T\in\Aut(\Gamma)$. Then the following statements are equivalent:
\begin{enumerate}
\item[{$(1)$}] $T\in\NC$.
\item[{$(2)$}] $T$ acts distally on $\Sub^c_\Gamma$. 
\item[{$(3)$}] $T$ acts distally on $\Sub_\Gamma$.
\item[{$(4)$}] $T^n=\Id$ for some $n\in\N$. 
\item[{$(5)$}] $T^n|_{\Gamma'}=\Id$ for some $n\in\N$, where $\Gamma'$ is a subgroup of finite index in $\Gamma$. 
\end{enumerate}
If $T=S|_\Gamma$ for some $S\in \Aut(G)$, then $(1-5)$ are equivalent to each of the following statements: 
\begin{enumerate}
\item[{$(6)$}] $S$ acts distally on $\Sub_G$. 
\item[{$(7)$}] $S$ is contained in a compact subgroup of $\Aut(G)$.
\end{enumerate}
Moreover, if $G$ has no compact factors, then $(1-7)$ are equivalent to the following: 
\begin{enumerate}
\item[{$(8)$}]$S^n=\Id$ for some $n\in\N$. 
\end{enumerate}
\end{thm}

\begin{proof} $(4)\implies (3)\implies (2) \implies (1)$ and $(4)\implies (5)$ are obvious. Now suppose $(5)$ holds. 
We show that (4) holds. Passing to a smaller subgroup of finite index if necessary, we may assume that $\Gamma'$ is 
normal in $\Gamma$ and that it is $T$-invariant. 

Since $\Gamma/\Gamma'$ is finite, replacing $n$ by a larger number if necessary, we may assume that 
${T^n}|_{\Gamma'}=\Id$ and $T^n$ acts trivially on $\Gamma/\Gamma'$. Without loss of any generality, we 
replace $T$ by $T^n$ and assume that $T|_{\Gamma'}$ is trivial and $T$ acts trivially on $\Gamma/\Gamma'$. 
We want to show that some power of $T$ is the identity map. 

Let $x\in\Gamma$. Then $T(x)=xy$ for some $y\in \Gamma'$. For any $g\in \Gamma'$, we have 
$xygy^{-1}x^{-1}=T(xgx^{-1})=xgx^{-1}$, and hence $ygy^{-1}=g$. Therefore, $y\in Z(\Gamma')$, 
the center of $\Gamma'$. By Lemma \ref{newl-center}, $Z(\Gamma)\cap\Gamma'$ is a subgroup 
of finite index in $Z(\Gamma')$. Let $m$ be the order of $Z(\Gamma')/(Z(\Gamma)\cap \Gamma')$ and let $k$ be the order of 
$\Gamma/\Gamma'$. Then $y^m\in Z(\Gamma)$ and we get that  $T^m(x)=xy^{m}\in x Z(\Gamma)$ and hence 
$T^m(x^k)=x^ky^{km}$. As $x^k\in\Gamma'$, we have that $T^m(x^{k})=x^k$, 
and hence $y^{km}=e$. Therefore, $T^{km}(x)=x$. Thus $T^{km}=\Id$ and (4) holds. 

Now we show that $(1)\implies (5)$. Suppose $T\in\NC$. Let $\Gamma'$ be a normal subgroup of finite index in $\Gamma$ 
as in Lemma \ref{newl}. That is, the set $R_g$ of roots of $g$ in $\Gamma'$ is finite for every $g\in\Gamma'$. Without loss of 
any generality, we may assume that $\Gamma'$ is $T$-invariant and $T|_{\Gamma'}\in\NC$. Note that $\Gamma'$, being a 
subgroup of finite index in $\Gamma$, is a lattice in $G$. Hence $\Gamma'$ is finitely generated, and we get from 
Proposition \ref{discr-subg} that $T^n|_{\Gamma'}=\Id$ for some $n\in\N$ and $(5)$ holds. That is, $(1-5)$ are equivalent.  

Let $S\in\Aut(G)$ and let $S|_\Gamma=T$. Then $(7)\implies (6)$ (see Lemma 2.4 in \cite{SY3} and the discussion before 
the lemma, or see Theorem 4.1 of \cite{SY3}, or Lemma \ref{aut-g} below). Note that $(6)\implies (3)$ is obvious. 
Now we prove that $(4)\implies (7)$. Suppose $T^n=\Id$ for some $n\in\N$. 
Since $G$ is semisimple, some power of $S$ is an inner automorphism of $G$. To prove (7), we may assume that 
$S$ itself is an inner automorphism of $G$. Let $s\in G$ be such that $S=\inn(s)$. Now $s\Gamma s^{-1}=\Gamma$ 
and from (4), we get that $s^l$ centralises $\Gamma$ for some $l\in\N$. Replacing $S$ by $S^l$, we may assume that 
$s\in Z_G(\Gamma)$, the centraliser of $\Gamma$ in $G$. Let $K$ be the largest compact connected normal subgroup 
of $G$ which is the product of all compact factors of $G$. 

If $K$ is trivial, then by Theorem 5.18 of \cite{Rag}, $Z_G(\Gamma)= Z(G)$, and hence $S=\Id$. That is, if  
$G$ has no compact factors, then (8) holds, and hence (7) also holds in this case. 
(Note that $(4)\implies (8)$ also follows directly from the Borel Density Theorem if $G$ has no compact factors.)

If $G$ is compact, then $\Aut(G)$ is compact as $\Inn(G)$ is a subgroup of finite index in $\Aut(G)$, hence (7) holds. 
Now suppose $G$ is not compact. Then $G=KG_1$ (almost direct product), where $G_1$ is a closed connected normal subgroup 
which is the product of all non-compact (simple) factors of $G$. Now $s=kh=hk$ for some $k\in K$ and $h\in G_1$. 
Let $\psi:G\to G/K$ be the natural projection. Then $\psi(\Gamma)$ is a lattice in $G/K$. As $G/K$ has no compact factors 
and as $\psi(s)$ centralises $\psi(\Gamma)$, we get as above that $\psi(s)\in Z(G/K)$, and since $\psi(s)=\psi(h)$, 
$hgh^{-1}g^{-1}\in K$ for all $g\in G$. Since $h\in G_1$, which is normal in $G$, we get
that $hgh^{-1}g^{-1}\in G_1\cap K$ which is a finite (central) subgroup of $G$. As $G$ is connected, the preceding assertion 
implies that $hgh^{-1}g^{-1}=e$ for all $g\in G$, and hence $h\in Z(G)$. Now $s=kh\in KZ(G)$ and $\inn(s)=\inn(k)$. 
As $k\in K$, we get that $\inn(s)$, and hence $S$ is contained in a compact subgroup of $\Aut(G)$. Therefore, $(7)$ holds. 
\end{proof} 

Note that Example \ref{ex1} shows that a connected simply connected solvable Lie group can admit an automorphism $T$ and
$T$-invariant lattices  $\Gamma_1$ and $\Gamma_2$ such that $\Gamma_1$ is a subgroup of finite index in $\Gamma_2$ and 
$T|_{\Gamma_1}=\Id$ but $T^n|_{\Gamma_2}\ne\Id$ for any $n\in\N$. This is unlike the case of lattices in a 
connected semisimple Lie group as shown by $(5)\implies (4)$ in Theorem \ref{newp}.

It would be interesting to study the distality of the actions of automorphisms of $\Gamma$ on $\Sub_\Gamma$
for a lattice $\Gamma$ in a general connected Lie group. 

\section{Distal Actions of Automorphisms on Sub$_{\boldsymbol G}$ for Certain Compact Groups and Nilpotent Groups}

In this section, for certain locally compact metrizable groups $G$ and $T\in\Aut(G)$, we characterise the distality of 
the $T$-action on $\Sub_G$ in terms of the compactness of the closure of the group generated by $T$ in $\Aut(G)$. 
It is shown in \cite{SY3} that if $\Aut(G)$ is endowed with the modified compact-open topology, then the map 
$\Aut(G)\times\Sub_G\to\Sub_G$ defined by $(T,H)\mapsto T(H)$, $T\in\Aut(G)$, $H\in\Sub_G$, is continuous, 
i.e.\  $\Aut(G)$ acts continuously on $\Sub_G$ by homeomorphisms (cf.\ \cite{SY3}, Lemma 2.4). For compact groups $G$, 
the modified compact-open topology is the same as the compact-open topology on $\Aut(G)$. For any connected Lie group 
$G$ with the Lie algebra $\G$, for a $T\in\Aut(G)$, there is a unique Lie algebra automorphism $\du T$ in $\GL(\G)$. 
Note that $\Aut(G)$ is isomorphic to a closed subgroup of $\GL(\G)$, (isomorphism is given by the map $T\mapsto\du T$). 
Therefore, $\Aut(G)$ is a Lie group whose topology is the same as the compact-open topology as well as the modified 
compact-open topology, (cf.\ \cite{Ar, Ho}). In general, if $\Aut(G)$ is endowed with the compact-open topology, 
then $\Aut(G)$ is a topological semigroup and the natural map $\Aut(G)\times G\to G$ is continuous, (see \cite{Str} for more 
details on topologies on $\Aut(G)$).

For a metric space $X$, a subset $\Omega$ of $\Homeo(X)$ is said to be equicontinuous at $x\in X$ if given $\epsilon>0$, 
there exists $\delta>0$ such that $\phi(B_\delta(x))\subset B_\epsilon(\phi(x))$, $\phi\in\Omega$, where $B_r(x)$ is 
the ball of radius $r$ centered at $x$ in $X$ for $r>0$. $\Omega$ is said to be equicontinuous on $X$ if $\Omega$ is 
equicontinuous at every $x\in X$. If $G$ is a locally compact first countable (metrizable) group with the identity $e$, then 
$G$ has a left invariant metric, and hence any $\Omega\subset \Aut(G)$ is equicontinuous at $x$ if and only if it is 
equicontinuous at $e$. Therefore, $\Omega$ is equicontinuous on $G$ if and only if given any neighbourhood $U$ of $e$, 
there exists a neighbourhood $V$ of $e$ such that $\phi(U)\subset V$ for all $\phi\in\Omega$. 

By Arzela-Ascoli Theorem (see e.g.\ \cite{Str}, Theorem 9.24), $\Omega\subset\Aut(G)$ is relatively compact in $\Aut(G)$ 
(with respect to the compact-open topology) if it is equicontinuous at $e$ and $\{\phi(x)\mid \phi\in \Omega\}$, the 
$\Omega$-orbit of $x$ is relatively compact in $G$, for every $x\in G$. The converse also holds since $G$ is locally compact 
and $\ol{\Omega}$ is compact, the action of $\ol{\Omega}$ on $G$ is uniformly continuous. The following useful version of 
Arzela-Ascoli Theorem for locally compact metrizable groups easily follows from above. 

\begin{lem} \label{aa}{\rm [Arzela-Ascoli Theorem]}
If $G$ is a locally compact first countable $($metrizable$)$ group. Let $\Aut(G)$ be the group of automorphisms of $G$ 
endowed with the compact-open topology. Let $\Omega$ be a subset of $\Aut(G)$. Then $\ol{\Omega}$ is compact in 
$\Aut(G)$ if and only if the following hold:
\begin{enumerate}
\item $\{\phi(x)\mid \phi\in \Omega\}$ is relatively compact in $G$, for every $x\in G$. 
\item $\Omega$ is equicontinuous at $e$.
\end{enumerate}
\end{lem}

The following  lemma will be useful. 

\begin{lem} \label{aut-g} Let $G$ be a locally compact first countable $($metrizable$)$ group. Let $H\subset \Aut(G)$ be a subgroup. 
If $H$ is relatively compact with respect to the compact-open topology, then
$\ol{H}$ is a compact group and $H$ acts distally on $\Sub_G$. 
\end{lem}

\begin{proof} Note that $\Aut(G)$ is a topological semigroup with respect to the compact-open topology 
(see e.g.\ Lemma 9.5 of \cite{Str}). Therefore, $\ol{H}$ is a compact semigroup and hence $\ol{H}$ is a compact group 
in $\Aut(G)$ with respect to the compact-open topology (see e.g.\ Theorem 30.6 of \cite{Str}). Note that on $\ol{H}$, 
the compact-open topology and the modified compact-open topology coincide. It follows from Lemma 2.4 of \cite{SY3} that 
the natural map $\ol{H}\times\Sub_G\to\Sub_G$, defined by the action of automorphisms of $G$ on $\Sub_G$, 
is continuous. As $\ol{H}$ is compact, it follows that $\ol{H}$, and hence $H$ acts distally on $\Sub_G$. 
\end{proof}

As observed in the proof of Lemma \ref{aut-g}, it follows that a subgroup $H$ of $\Aut(G)$ is compact with respect to 
the compact-open topology, if and only if it is compact with the respect to the modified compact-open topology. Moreover, 
a subgroup $H$, which is compact (in the compact-open topology), is a compact topological group. Henceforth, $\Aut(G)$ 
is endowed with the compact-open topology, and for compact subgroups of $\Aut(G)$, we will not specify the topology. 

For a totally disconnected locally compact group $G$, $T\in\Aut(G)$ is distal if and only if $G$ has arbitrarily small compact 
open $T$-invariant subgroups, (this follows from Proposition 2.1 of \cite{JR} together with the `Note added in proof' in \cite{JR}). 
Moreover if $G$ is metrizable, then the above implies that, $T$ is distal if and only if $\{T^n\}_{n\in\Z}$ is equicontinuous (at $e$).  
We now get the following characterisation for compact totally disconnected groups. 

\begin{prop} \label{cpt-distal}
Let $G$ be a compact totally disconnected first countable $($metrizable$)$ group and let $T\in\Aut(G)$. Then the following are equivalent:
\begin{enumerate}
\item[{$(1)$}] $T$ acts distally on $G$.
\item[{$(2)$}] $T$ acts distally on $\Sub_G$.
\item[{$(3)$}] $T$ is contained in a compact subgroup of $\Aut(G)$. 
\end{enumerate}
\end{prop}

\begin{proof}
Here, $(3)\implies (2)$ follows from Lemma \ref{aut-g}. As $G$ is totally disconnected, $(2)\implies (1)$ follows from 
Theorem 3.6 of \cite{SY3}. It is enough to show that $(1)\implies (3)$. Suppose $T$ acts distally on $G$. Let 
$\Omega_T=\{T^n\}_{n\in\Z}$. By Proposition 2.1 of \cite{JR}, $\Omega_T$ is equicontinuous (at $e$). Also, since 
$G$ is compact, the $\Omega_T$-orbit of $x$ is relatively compact for every $x\in G$. By Lemma \ref{aa},  
$\Omega_T$ has compact closure in $\Aut(G)$. Hence $\ol{\Omega_T}$ is a compact group 
(see e.g.\ Theorem 30.6 of \cite{Str}).
\end{proof} 

Note that Proposition \ref{cpt-distal} also holds for a  non-compact totally disconnected (additive) group 
$G=\Q_p^n$, ($n\in\N$), a $p$-adic vector space, and $T\in \GL(n,\Q_p)$ (where $p$ is a prime). This follows from 
Lemma 2.1 of \cite{SY2} and Lemma \ref{aut-g} above together with the fact that $\GL(n,\Q_p)$ is a (metrizable) 
topological group and its topology is the same as the (modified) compact-open topology. 

The following generalises Theorem 4.1 of \cite{SY3} in the case of connected nilpotent Lie groups to 
all compactly generated nilpotent groups. Note that in any connected nilpotent group $G$, the unique 
maximal compact subgroup $K$ is connected, abelian and central in $G$ (as $G$ is Lie projective) and 
its torsion group is dense in $K$. Therefore, such a $G$ is torsion-free if and only if it is a simply connected 
nilpotent Lie group, (equivalently, it has no nontrivial compact subgroup). Any compactly generated nilpotent 
Lie group is torsion-free if and only if its maximal compact subgroup is trivial. 

\begin{thm} \label{cg-nilp-distal} 
Let $G$ be a locally compact metrizable compactly generated nilpotent  group such that 
$G^0$ is torsion-free and let $T\in\Aut(G)$. 
Then the following are equivalent.
\begin{enumerate}
\item[{$(1)$}] $T$ acts distally on $\Sub_G$.
\item[{$(2)$}] The closure of the group generated by $T$ in $\Aut(G)$ is a compact group.
\end{enumerate}
Moreover, if $G$ as above is a Lie group $($with not necessarily finitely many connected components$)$, then 
the following are equivalent and they also equivalent to statements $(1-2)$ above. 
\begin{enumerate}
\item[{$(3)$}] $T\in\NC$.
\item[{$(4)$}] $T$ acts distally on $\Sub^a_G$.
\end{enumerate}
\end{thm}

\begin{proof} Let $\Omega_T= \{T^n\}_{n\in\Z}$. If $\ol{\Omega_T}$ is compact, then it is a compact group 
(cf.\ \cite{Str}, Theorem 30.6).
Now using Lemma \ref{aut-g}, we get that $(2)\implies (1)$. We know that $(1)\implies (4)\implies (3)$. Suppose $T\in\NC$. 
Since $G$ is strongly root compact, by Lemma \ref{mt}, $\{x\in G\mid G_x\mbox{ is closed}\}\subset M(T)$. 
Since $G$ is compactly generated and nilpotent, it has a unique maximal compact group $K$ such that $G/K$ is 
a compactly generated torsion-free Lie group and all its cyclic subgroups are discrete. Now if $x\not\in K$, we get that 
$G_x$ is closed, and hence $x\in M(T)$. As $K$ is $T$-invariant, we have $K\subset M(T)$, and hence $G=M(T)$. 
This implies that $\Omega_T$ satisfies the condition (1) of Lemma \ref{aa}. 

 As $G^0$ is a simply connected nilpotent Lie group, by
Theorem 4.1 of \cite{SY3}, we get that $T|_{G^0}$ generates a relatively compact group in $\Aut(G^0)$. This implies that 
$\{(T|_{G^0})^n\}_{n\in\Z}$ is equicontinuous on $G^0$. 

Suppose $G$ is a Lie group. As $G^0$ is open, the preceding assertion implies that 
$\Omega_T$ is equicontinuous at $e$ and $\Omega_T$  satisfies the condition (2) of Lemma \ref{aa}. 
Therefore, $(3)\implies (2)$, and hence $(1-4)$ are equivalent for a Lie group $G$. 

Suppose $G$ is not a Lie group and suppose (1) holds. Then $T\in\NC$ and $\Omega_T$ satisfies the condition (1) of 
Lemma \ref{aa} as shown above. As $T$ acts distally on $\Sub_G$, by Theorem 3.6 of \cite{SY3}, $T$ is distal. As $G$ 
is compactly generated and nilpotent, it is distal and 
by Corollary \ref{T-Lie}, $G$ is $T$-Lie projective. Therefore, there exist compact open $T$-invariant normal subgroups 
$K_n$ such that $G/K_n$ is a Lie group, $K_n\subset K_{n+1}$, $n\in\N$, and $\bigcap_n K_n=\{e\}$. As $G^0$ has 
no nontrivial compact subgroup, $G_n:=G^0\times K_n$ are open $T$-invariant subgroups of $G$ such that $T(G^0)=G^0$ and 
$T(K_n)=K_n$, $n\in\N$. We know from above that $\{(T|_{G^0})^n\}_{n\in\Z}$ is equicontinuous on $G^0$. Let 
$\{W_m\}_{m\in\N}$ be a neighbourhood basis of the identity $e$ in $G^0$ such that $T^k(W_{m+1})\subset W_m$ 
for all $k\in \Z$ and $m\in \N$. Then $\{K_n\times W_m\mid m,n\in\N\}$ is a neighbourhood basis of the identity 
$e$ in $G$. As $K_n$ are $T$-invariant, it follows that $\Omega_T$ is equicontinuous at $e$ and it satisfies 
the condition (2) of Lemma \ref{aa}. Hence ${\ol\Omega_T}$ is a compact group in $\Aut(G)$ and (2) holds. 
\end{proof}

Note that if $\Gamma$ is a locally compact compactly generated nilpotent group without any nontrivial compact subgroup, 
then $\Gamma$ embeds  in a connected simply connected nilpotent Lie group $G$ as a closed co-compact subgroup 
and any automorphism of $\Gamma$ extends to a unique automorphism of $G$ (cf.\ \cite{Rag}). Note also that any 
closed subgroup of a simply connected nilpotent group is compactly generated. We now have the following corollary 
which can be viewed as an extension of Corollary \ref{lattice-nilp}. 
 
\begin{cor} \label{closed-nil}
Let $G$ be a connected simply connected nilpotent Lie group. Let $\Gamma$ be a closed co-compact subgroup of $G$. 
Let $T\in \Aut(G)$ be such that $T(\Gamma)=\Gamma$. Then $(1-6)$ of Corollary \ref{lattice-nilp} are equivalent and 
they are also equivalent to the following: $T$ is contained in a compact subgroup of $\Aut(G)$.
\end{cor}

\begin{proof} Note that it is enough to show that if $T|_\Gamma\in\NC$, then $T$ is contained in a compact subgroup 
of $\Aut(G)$. Let $T|_\Gamma\in\NC$. By Theorem \ref{cg-nilp-distal}, $T|_\Gamma$ is contained in a compact subgroup 
of $\Aut(\Gamma)$. Here, $\exp:\G\to G$ is a homeomorphism with $\log$ as it is inverse. Let $\du T:\G\to\G$ be 
the Lie algebra automorphism corresponding to $T$. Since $\Gamma$ is co-compact, it follows that $\log(\Gamma)$ 
generates $\G$ as a vector space and we also have that $\{\du T^n(g)\}_{n\in\Z}$ is relatively compact for all 
$g\in\log(\Gamma)$, and hence it is relatively compact for all  $g\in\G$. This implies that $\du T$ is contained in a compact 
subgroup of $\GL(\G)$. As $\Aut(G)$ is a closed subgroup of $\GL(\G)$, $T$ is contained in a compact subgroup of $\Aut(G)$. 
\end{proof}

Note that the action of $G$ on $\Sub_G$ is the same as the action of $\Inn(G)$ on $\Sub_G$, where $\Inn(G)$ is the group 
of inner automorphisms of $G$. For subgroups $H_1$ and $H_2$ of $G$, let $[H_1,H_2]$ denote the subgroup generated by 
$\{h_1h_2h_1^{-1}h_2^{-1}\mid h_1\in H_1, h_2\in H_2\}$. Recall that for a subgroup $H$ of $G$, $Z_G(H)$ denotes the 
centraliser of $H$ in $G$. The following theorem is an analogue of Corollary 4.5 of \cite{SY3} in the case of certain disconnected 
nilpotent groups. 

\begin{thm} \label{distal-gp}
Let $G$ be a locally compact  metrizable compactly generated nilpotent group and let $K$ be the unique maximal compact 
$($normal$)$ subgroup of $G$.  Then the following are equivalent:
\begin{enumerate}
\item[{$(1)$}] Every inner automorphism of $G$ acts distally on $\Sub_G$.
\item[{$(2)$}] $G$ acts distally on $\Sub_G$. 
\item[{$(3)$}] $\ol{\Inn(G)}$ is a compact subgroup of $\Aut(G)$. 
\item[{$(4)$}] $G/K$ is abelian and $G=Z_G(G^0)$. 
\end{enumerate}
In case $G$ is a torsion-free Lie group, then $(1-4)$ are equivalent to the following: 
\begin{enumerate}
\item[{$(5)$}] $G$ is abelian. 
\end{enumerate}
\end{thm}

\begin{proof} Here $(3)\implies (2)$ follows from Lemma \ref{aut-g}.  
It is obvious that $(2)\implies (1)$. We now show that $(1)\implies (4)$. Suppose (1) holds. 
Since $K$ is the unique maximal compact subgroup of $G$, $K$ is normal in $G$ and $G/K$ is 
a compactly generated nilpotent Lie group without any nontrivial compact subgroup. By Lemma 3.1 of \cite{SY3}, 
every inner automorphism of $G/K$ acts distally on $\Sub_{G/K}$. To prove that $G/K$ is abelian, we may assume 
that $G$ is a Lie group without any nontrivial compact subgroup and show that it is abelian.  If possible, suppose 
$G$ is not abelian. Let $Z=Z(G)$, the centre of $G$. Then $Z\ne G$, and since $G$ is nilpotent, there exists 
a closed subgroup $Z_1=\{g\in G\mid xgx^{-1}g^{-1}\in Z \mbox{ for all } x\in G\}$ such that $Z_1\supsetneq Z\supsetneq\{e\}$. 
Let $y\in Z_1$ be such that $y\not\in Z$. Then there exists $x\in G$ such that $xyx^{-1}=yz$ for some nontrivial 
$z\in Z$. Now $\inn(x)(y)=yz$, and $\inn(x)$ acts trivially on $Z$ which has no nontrivial compact subgroup. 
Let $G_y$ be the subgroup generated by $y$ in $Z_1$. Here, $xy^nx^{-1}=y^nz^n$ and since $G$ is torsion-free, 
we have that no element of $G_y$ is stabilised by $\inn(x)$. By Lemma 3.12 of \cite{SY3}, $\inn(x)\not\in\NC$. 
In particular, $\inn(x)$ does not act distally on $\Sub_G$. This contradicts the statement in (1), and hence 
$G$ is abelian. This implies the first assertion in (4).

If $G$ is a torsion-free Lie group, we have that $K$ is trivial since the set of torsion elements is dense in $K$. 
Hence, the above shows that for such a $G$, $(1)\implies (5)$ and also $(4)\implies (5)$. Since $(5)\implies (4)$ 
and $(5)\implies (3)$, all the statements $(1-5)$ are equivalent for such a torsion-free Lie group $G$.

Now we show that $G=Z_G(G^0)$, i.e.\ we show that $G^0$ is central in $G$. If $G^0$ is trivial, then $G=Z_G(G^0)$. 
Now suppose $G^0\ne\{e\}$. We know from Proposition \ref{centraliser-nilp} that $K$ centralises $G^0$. 

Suppose $G$ is a Lie group. We know that $C=K\cap G^0$ is the maximal compact subgroup of $G^0$, $C$ is connected 
and central in $G^0$, and by Proposition \ref{centraliser-nilp}, $C$ is central in $KG^0$. 

We first show that $C$ is central in $G$. Since $C$ is characteristic in $G^0$, $C$ is normal in $G$. As $G/K$ is abelian 
and $G^0$ is normal, we have that $[G,G^0]\subset K\cap G^0=C$. Suppose $C$ is not central in $G$ and 
suppose $x\in G$ does not centralise $C$. Then $x\not\in KG^0$. As $G/K$ has no nontrivial compact subgroup, 
$xK$ generates a discrete infinite subgroup in $G/K$, and hence the cyclic group $G_x$ generated by $x$ in $G$ 
is discrete and infinite. From above, $[G_x,C]\ne \{e\}$. Let $C_0=C$ and  $C_k=\ol{[G_x,C_{k-1}]}$, $k\in\N$. 
Here, $C_1\ne \{e\}$, $C_k\subset C_{k-1}\subset C$, $xC_kx^{-1}=C_k$ and $C_k$ is a compact subgroup of $C$, $k\in\N$. 
Since $C$ is connected, we get that $C_1$, and hence each $C_k$ is connected. Moreover, $\inn(x)$ acts 
trivially on $C_{k-1}/C_k$, $k\in\N$. As $G$ is nilpotent, $C_l\ne\{e\}$ and $C_{l+1}=\{e\}$ for some $l\in\N$. 

Since $C$ is a connected abelian Lie group, the above implies that the action of $\inn(x)$ on $C$ is unipotent 
(i.e.\ the eigenvalues of $\Ad(x)$ on the Lie algebra of $C$ are all equal to 1). As $\inn(x)$ acts distally on $\Sub_C$, 
by Proposition 4.2 of \cite{SY3}, we get that $\inn(x)$ acts trivially on $C$, (one can also directly argue as in Step 1 of 
the proof of Proposition 4.2 of \cite{SY3} to conclude that $\inn(x)|_{C}=\Id$).  This leads to a contradiction, 
and hence we get that $x$ centralises $C$. Since this holds for all $x\in G$ we get that $G=Z_G(C)$. 

Now suppose $G^0\ne C$. By Corollary 4.5 of \cite{SY3}, $G^0=\R^n\times C$ for some $n\in\N$. Suppose 
$G\ne Z_G(G^0)$. There exists $x\in G$ which does not centralise $G^0$. Then $x\not\in KG^0$ as $KG^0$ centralises 
$G^0$. Note that as $G/K$ is abelian, we have that $[G,G^0]\subset C$ and $xgx^{-1}\in gC$ for all $g\in\R^n$. Therefore, 
$\inn(x)$ acts trivially on $G^0/C$, and hence the action of $\inn(x)$ on $G^0$ is unipotent. By Proposition 4.2 of \cite{SY3}, 
$\inn(x)$ acts trivially on $G^0$, (one can also directly conclude this by arguing as in the latter part of Step 2 of the proof 
of Proposition 4.2 of \cite{SY3}). This leads to a contradiction, and hence $x\in Z_G(G^0)$. Since this holds for all $x\in G$, 
we have that $G=Z_G(G^0)$.

Now suppose $G$ is not a Lie group. Since $G$ is nilpotent, it is distal, and since it is compactly generated, 
locally compact and metrizable, by Theorem \ref{distal-str}, 
$G$ is a projective limit of Lie groups $G/K_n$, where $K_n\subset K$, $n\in\N$, and $\bigcap_n K_n=\{e\}$. 
By Lemma 3.1 of \cite{SY3}, every inner automorphism of $G/K_n$ also acts distally on $\Sub_{G/K_n}$, $n\in\N$. 
As $(G/K_n)^0=G^0K_n/K_n$ is a Lie group, we get from above that $G/K_n=Z_{G/K_n}(G^0K_n/K_n)$. This implies that 
$[G,G^0]\subset K_n$ for all $n$, and hence $[G,G^0]\subset \bigcap_n K_n=\{e\}$. Therefore, $G=Z_G(G^0)$. 
This completes the proof of $(1)\implies (4)$.

Now suppose $(4)$ holds. We show that (3) holds. Since $G/K$ is abelian, for every $g\in G$, the $\Inn(G)$-orbit of $g$ 
is contained in $gK$. Therefore, (1) of Lemma \ref{aa} is satisfied for $\Omega=\Inn(G)$. Now we show that $\Inn(G)$ 
is equicontinuous at $e$. By Theorem \ref{distal-str}, $G$ is Lie projective, and hence has compact normal subgroups 
$K_n$ such that $G/K_n$ is a Lie group, $n\in\N$, and $\bigcap_n K_n=\{e\}$. Therefore, $G^0K_n$ is open in $G$. 
Let $\{U_n\}_{n\in\N}$ be a neighbourhood basis of the identity $e$ in $G^0$. Since $K$ centralises $G^0$, 
$K_nU_n=U_nK_n$, $n\in\N$, and $\{K_nU_n\}_{n\in\N}$ is a neighbourhood basis of the identity $e$ in $G$. As 
$G=Z_G(G^0)$, we have that for all $x\in G$, $xK_nU_nx^{-1}=K_nU_n$, $n\in\N$. Therefore, $\Inn(G)$ is equicontinuous at 
$e$ and by Lemma \ref{aa}, $\ol{\Inn(G)}$ is relatively compact in $\Aut(G)$, and hence it is a compact group and (3) holds. 
Therefore, $(1-4)$ are equivalent. 
\end{proof}

Note that in Theorem \ref{distal-gp}, (5) is not equivalent to $(1-4)$ in general. There exist metrizable compact non-abelian 
totally disconnected nilpotent groups $G$; e.g.\ take $G$ to be a subgroup of strictly upper triangular matrices in 
$\SL(3,\Z_p)$, where $\Z_p$ is the ring of $p$-dic integers in $\Q_p$ for a primes $p$. For such a $G$, the inner 
automorphisms act distally on $\Sub_G$, and hence Theorem \ref{distal-gp} $(1-4)$ obviously hold for $G$ 
due to Proposition \ref{cpt-distal}. 

\bigskip
\noindent{\bf Acknowledgements} R.\ Shah would like to acknowledge the MATRICS research grant from DST-SERB, 
Govt.\ of India. R.\ Palit would like to acknowledge the CSIR-JRF research fellowship from CSIR, Govt.\ of India. 
The authors would like to thank the anonymous referee for useful comments and for a sketch of the proofs for 
Lemma \ref{newl} and Theorem \ref{newp} for lattices in a connected semisimple Lie group. 

\bibliographystyle{amsplain}

\bigskip\medskip
\noindent{
\advance\baselineskip by 2pt
\begin{tabular}{ll}
Rajdip Palit & \hspace*{2.5cm}Riddhi Shah \\
School of Physical Sciences & \hspace*{2.5cm}School of Physical Sciences\\
Jawaharlal Nehru University & \hspace*{2.5cm}Jawaharlal Nehru University\\
New Delhi 110 067, India & \hspace*{2.5cm}New Delhi 110 067, India\\
rajdip1729@gmail.com & \hspace*{2.5cm}rshah@jnu.ac.in \\
{ }& \hspace*{2.5cm}riddhi.kausti@gmail.com
\end{tabular}}

\end{document}